\documentclass[11pt]{article}

\usepackage{amsmath,amsthm,amssymb,amsfonts,bbm,bm}
\usepackage{epsfig,epsf,url,enumitem}
\usepackage{etoolbox}
\usepackage{graphicx}
\usepackage[hidelinks]{hyperref}
\usepackage{fullpage}
\usepackage{cite}
\usepackage[small,bf]{caption}
\usepackage{array}
\newcolumntype{C}[1]{>{\centering\let\newline\\\arraybackslash\hspace{0pt}}m{#1}}

\usepackage{longtable}

\DeclareMathOperator*{\minimize}{minimize}

\usepackage{amsbsy}
\usepackage{bbm}

\newcommand{\svx}{\mathcal{S}}
\newcommand{\cvx}{\mathcal{F}}

\newcommand{\tp}{\mathsf{T}}
\newcommand{\norm}[1]{\left\|#1\right\|}

\newcommand{\field}[1]{\mathbb{#1}}
\newcommand{\R}{\field{R}}

\newcommand{\ee}{\,\mathbb{E} }

\newcommand{\bmtx}{\begin{bmatrix}}
\newcommand{\emtx}{\end{bmatrix}}
\newcommand{\bsmtx}{\left[ \begin{smallmatrix}} 
\newcommand{\esmtx}{\end{smallmatrix} \right]} 
\newcommand{\bmatarray}[1]{\left[\begin{array}{#1}}
\newcommand{\ematarray}{\end{array}\right]}

\newtheorem{remark}{Remark}

\newtheorem{lemma}{Lemma}
\newtheorem{theorem}{Theorem}
\newtheorem{corollary}{Corollary}
\newtheorem{definition}{Definition}
\newtheorem{proposition}{Proposition}

\makeatletter
\newcommand*{\rom}[1]{\expandafter\@slowromancap\romannumeral #1@}
\makeatother

\newcommand\EnumPrefix{}

\newlist{senenum}{enumerate}{10}
\setlist[senenum]{label=\arabic*.,ref=\EnumPrefix,leftmargin=*}

\AtBeginEnvironment{lemmalist}{\renewcommand\EnumPrefix{\thelemmalist.\arabic*}}

\numberwithin{equation}{section}

\newcommand{\remove}[1]{}

\title{Analysis of  Biased Stochastic Gradient Descent Using\\ Sequential Semidefinite Programs}

\author{
  Bin Hu\thanks{B.~Hu is with the Coordinated Science Laboratory and the Department of Electrical and Computer Engineering, University of Illinois at Urbana--Champaign, Email: binhu7@illinois.edu} \and
Peter Seiler\thanks{P. Seiler is with the Department of Electrical Engineering and Computer Science, University of Michigan, Ann Arbor,  Email: pseiler@umich.edu} \and
  Laurent Lessard\thanks{ L.~Lessard is with the Wisconsin Institute for Discovery and the Department of Electrical and Computer Engineering, University of Wisconsin--Madison, Email: laurent.lessard@wisc.edu} 
}

\date{}

\begin{document}

\maketitle

\begin{abstract}
We present a convergence rate analysis for biased stochastic gradient descent (SGD), where individual gradient updates are corrupted by computation errors.
We develop stochastic quadratic constraints to formulate a small linear matrix inequality (LMI) whose feasible points lead to  convergence bounds of biased SGD. Based on this LMI condition, we develop a sequential minimization approach to analyze the intricate trade-offs that couple stepsize selection, convergence rate, optimization accuracy, and robustness to gradient inaccuracy. We also provide feasible points for this LMI and obtain theoretical formulas that quantify the convergence properties of biased SGD under various assumptions on the loss functions.
\end{abstract}

\section{Introduction}
Empirical risk minimization (ERM) is a prevalent topic in machine learning research~\cite{bubeck2015, teo2007}. Ridge regression, $\ell_2$-regularized logistic regression,  and support vector machines (SVM) can all be formulated as the following ERM problem
\begin{align}\label{eq:opt}
\min_{x\in \R^p} g(x)= \frac{1}{n}\sum_{i=1}^n f_i(x),
\end{align}
where $g:\R^p\rightarrow \R$ is the objective function. Stochastic gradient descent (SGD) \cite{bottou2010, Bottou2004, robbins1951} has been widely used for ERM to exploit redundancy in the training data. The SGD method applies the update rule
\begin{align}
\label{eq:SGD}
x_{k+1}=x_k-\alpha_k u_k,
\end{align}
where $u_k = \nabla f_{i_k}(x_k)$ and the index $i_k$ is uniformly sampled from $\{1, 2, \ldots, n\}$ in an independent and identically distributed (IID) manner. 
The convergence properties of SGD are well understood.   Under strong convexity of $g$ and smoothness of $f_i$, SGD with a diminishing stepsize converges sublinearly, while SGD with a constant stepsize converges linearly to a ball around the optimal solution \cite{feyzmahdavian2014, Nedic2001, moulines2011, needell2014}.  In the latter case, epochs can be used to balance convergence rate and optimization accuracy. Some recently-developed stochastic methods such as SAG~\cite{Roux2012, Schmidt2013}, SAGA~\cite{defazio2014}, Finito~\cite{defazio2014finito}, SDCA~\cite{shalev2013}, and  SVRG~\cite{johnson2013} converge linearly with low iteration cost when applied to~\eqref{eq:opt}, though SGD is still popular because of its simple iteration form, low memory footprint, and nice generalization property. SGD is also commonly used as an initialization for other algorithms~\cite{Roux2012, Schmidt2013}.

In this paper, we present a general analysis for \textit{biased SGD}. This is a version of SGD where the gradient updates $\nabla f_{i_k}(x_k)$ are corrupted by additive as well as multiplicative noise. In practice, such errors can be introduced by sources such as: inaccurate numerical solvers, digital round-off errors, quantization, or sparsification. The biased SGD update equation is given by
\begin{align}
\label{eq:SGerror}
x_{k+1}=x_k-\alpha_k (u_k+ e_k).
\end{align}
Here, $u_k = \nabla f_{i_k}(x_k)$ is the individual gradient update and $e_k$ is an error term. We consider the following error model, which unifies the error models in \cite{berts2002}:
\begin{align}\label{eq:def_ykhk0}
\|e_k\|^2 \le \delta^2 \| u_k \|^2+c^2,
\end{align}
where $\delta \ge 0$ and $c\ge 0$ bound the relative error and the absolute error in the oracle computation, respectively. If $\delta=c=0$, then $e_k=0$ and we recover the standard SGD setup. 
The model \eqref{eq:def_ykhk0} unifies the error models in \cite{berts2002} since:
\begin{enumerate}
\item If $c=0$, then~\eqref{eq:def_ykhk0} reduces to a relative error model, i.e. 
\begin{align}\label{eq:def_ykhk}
\|e_k\| \le \delta \| u_k \|
\end{align}
\item If $\delta=0$, then $e_k$  is a bounded absolute error, i.e. 
\begin{align}\label{eq:def_ykhk2}
\|e_k\| \le c
\end{align}
\end{enumerate}
We assume that both $\delta$ and $c$ are known in advance. We make no assumptions about how $e_k$ is generated, just that it satisfies~\eqref{eq:def_ykhk0}. Thus, we will seek a worst-case bound that holds regardless of whether $e_k$ is random, set in advance, or chosen adversarially.

Suppose the cost function $g$ admits a unique minimizer $x_\star$.
For standard SGD (without computation error), $u_k$ is an unbiased estimator of $\nabla g(x_k)$. Hence under many circumstances,  one can control the final optimization error $\|x_k-x_\star\|$ by decreasing the stepsize $\alpha_k$. Specifically, suppose $g$ is $m$-strongly convex.  Under various assumptions on $f_i$, one can prove the following typical bound for standard SGD with a constant stepsize $\alpha$ \cite{moulines2011,
  needell2014, bottou2018optimization}:
\begin{align}
\label{eq:constMainBound}
\ee \|x_k-x_\star\|^2\le \rho^{2k}\mathbb{E}\|x_0-x_\star\|^2+H_\star
\end{align}
where $\rho^2=1-2m\alpha+O(\alpha^2)$ and
$H_\star=O(\alpha)$. By decreasing stepsize $\alpha$, one can control the final optimization error $H_\star$ at the price of slowing down the convergence rate $\rho$. 
The convergence behavior of biased SGD is different. Since the error term $e_k$ can be chosen adversarially, the sum $(u_k+e_k)$ may no longer be an unbiased estimator of $\nabla g(x_k)$. The error term $e_k$ may introduce a bias which cannot be overcome by decreasing stepsize $\alpha$. Hence the final optimization error in biased SGD heavily depends on the error model of~$e_k$.
 In this paper, we quantify the convergence properties of  biased SGD \eqref{eq:SGerror} with the error model \eqref{eq:def_ykhk0} using worst-case analysis.

\paragraph{Main contribution.}
The main novelty of this paper is that our analysis simultaneously addresses  the relative error and the absolute error in the gradient computation.
 We formulate a linear matrix inequality (LMI) that directly leads to convergence bounds of biased SGD and couples the relationship between $\delta$, $c$, $\alpha_k$ and the assumptions on $f_i$. This convex program can be solved both numerically and analytically to obtain various convergence bounds for biased SGD. Based on this LMI, we develop a sequential minimization approach that can analyze biased SGD with an arbitrary time-varying stepsize. We also obtain analytical rate bounds in the form of \eqref{eq:constMainBound} for biased SGD with constant stepsize.  However, our bound requires $\rho^2=1-\frac{m^2-\delta^2 \tilde M}{m}\alpha+O(\alpha^2)$~\footnote{When $\delta=c=0$, this rate bound does not reduce to  $\rho^2=1-2m\alpha+O(\alpha^2)$. This is due to the inherent differences between the analyses of biased SGD and the standard SGD. See Remark~\ref{rem:differentiability} for a detailed explanation.} and $H_\star=\frac{c^2+2\delta^2 G^2}{m^2-\delta^2\tilde M}+O(\alpha)$ where $\tilde M$ and $G^2$ are some prescribed constants determined by the assumptions on $f_i$.
Based on this result, there is no way to shrink $H_\star$ to $0$. This is consistent with our intuition since the gradient estimator as well as the final optimization result can be biased. We show that this ``uncontrollable" biased optimization error is $\frac{c^2+2\delta^2 G^2}{m^2-\delta^2\tilde M}$. 
The resultant analytical rate bounds highlight the design trade-offs for biased SGD.


The work in this paper complements the ongoing research on stochastic optimization methods, which mainly focuses on the case where the oracle computation is exact.  The stepsize selection in biased SGD must address the trade-offs between speed, accuracy, and inexactness in the oracle computations.  Our analysis brings new theoretical insights for understanding such trade-offs  in the presence of biased gradient computation.  It is also worth mentioning that the robustness of full gradient methods with respect to gradient inexactness has been extensively studied \cite{d2008smooth, schmidt2011, devolder2014first}.  However, addressing a unified error model that combines the absolute error and the relative error is still non-trivial. Our analysis complements the existing results in  \cite{d2008smooth, schmidt2011, devolder2014first} by providing a unified treatment of  the error model \eqref{eq:def_ykhk0}.
Notice that it is important to include the relative error model in the analysis since it  covers the numerical round-off error as a special case.  If one treats the round-off error as an absolute error with time-varying $c_k$, then the specific value of $c_k$ will depend on the state $x_k$ and can not be fixed beforehand. In contrast, if one models the round-off error as a relative error, the value $\delta$ can be fixed as a constant beforehand.

The approach taken in this paper can be viewed as a stochastic extension of the work in \cite{Lessard2014, nishihara2015} that analyzes the linear convergence rates of deterministic optimization methods (gradient descent, Nesterov's method, ADMM, etc.) using quadratic constraints and semidefinite programs.  Notice that the analysis for (deterministic) biased gradient descent in \cite{Lessard2014} is numerical. In this paper, we derive analytical formulas quantifying the convergence properties of the biased SGD. 
It is worth mentioning that one can combine jump system theory with quadratic constraints
to analyze SAGA, Finito, and SDCA in a unified manner \cite{hu17b}. However, the analysis in \cite{hu17b} does not directly address the trade-offs between the convergence speed $\rho^2$ and the optimization error $H_\star$, and cannot be easily tailored for biased SGD.
Another related line of work that uses semidefinite programs to analyze optimization methods is built upon the idea of formulating worst-case analysis as the so-called performance estimation problem (PEP) \cite{drori2014, taylor2017,taylor2017exact}. It is recognized that there is a fundamental connection between the quadratic constraint approach and the PEP framework \cite{taylor18a}. 
Recently, the PEP framework in \cite{drori2014, taylor2017,taylor2017exact} has been extended for the stochastic setup \cite{taylor19a}. In addition, it is known that the PEP approach can also be applied to study the bias in the (deterministic) gradient descent method \cite{de2017worst}.
 It is possible to extend the results in \cite{de2017worst, taylor19a} for a PEP-based analysis of biased SGD. This is an interesting topic for future research.

The rest of the paper is organized as follows.  In Section \ref{sec:generalAnalysis}, we formulate LMI testing conditions for  convergence analysis of biased SGD. The resultant LMIs are then solved sequentially, yielding recursive convergence bounds for biased SGD. In Section \ref{sec:fixed_alpha}, we simplify the analytical solutions of  the resultant sequential LMIs and derive analytical rate bounds in the form of \eqref{eq:constMainBound} for biased SGD with a constant stepsize.  Our results highlight various design trade-offs for biased SGD. Finally, we show how existing results on standard SGD (without gradient computation error) can be recovered using our proposed LMI approach, and discuss how a time-varying stepsize can potentially impact the convergence behaviors of biased SGD (Section \ref{sec:further}).

\subsection{Notation}

The $p\times p$ identity matrix and the $p \times p$ zero matrix are denoted as $I_p$ and $0_p$, respectively. The subscript $p$ is occasionally omitted when the dimensions are clear by the context.  When a
matrix $P$ is negative semidefinite, we will use the notation $P \preceq 0$.
The Kronecker product of two matrices $A$ and $B$ is denoted $A \otimes B$.

\begin{definition}[Smooth functions]\label{def:smooth}
A differentiable function  $f:\R^p \to \R$ is $L$-smooth for some $L>0$ if the following inequality is satisfied:
\[
\norm{\nabla f(x) - \nabla f(y)} \le L\norm{x-y}
\qquad\text{for all }x,y\in\R^p.
\] 
\end{definition}

\begin{definition}[Convex functions]\label{def:cvx}
Let $\cvx(m,L)$ for $0 \le m \le L \le \infty$ denote the set of differentiable functions $f:\R^p \to \R$ satisfying the following inequality for all $x,y\in\R^p$.
\begin{align}
\label{eq:gradient3}
\bmtx x-y \\ \nabla f(x)-\nabla f(y) \emtx^\tp \bmtx -2m I_p & (1+\tfrac{m}{L}) I_p \\  (1+\tfrac{m}{L}) I_p & -\tfrac{2}{L} I_p\emtx  
\bmtx x-y \\ \nabla f(x)-\nabla f(y)\emtx\ge 0.
\end{align}
\end{definition}

\noindent
Note that $\cvx(0,\infty)$ is the set of all convex functions, $\cvx(0,L)$ is the set of all convex $L$-smooth functions, $\cvx(m,\infty)$ with $m>0$ is the set of all $m$-strongly convex functions, and $\cvx(m,L)$ with $m>0$ is the set of all $m$-strongly convex and $L$-smooth functions. If $f \in \mathcal{F}(m,L)$ with $m>0$, then $f$ has a unique global minimizer.

\begin{definition}\label{def:svx}
Let $\svx(m,L)$ for $0 \le m \le L \le \infty$ denote the set of differentiable functions $g:\R^p \to \R$ having some global minimizer $x_\star \in \R^p$ and satisfying the following inequality for all $x,y\in\R^p$.
\begin{align}
\label{eq:gradientNL}
\bmtx x-x_\star \\ \nabla g(x)\emtx^\tp \bmtx -2m I_p & (1+\tfrac{m}{L}) I_p \\  (1+\tfrac{m}{L}) I_p & -\tfrac{2}{L} I_p\emtx  
\bmtx x-x_\star \\ \nabla g(x)\emtx\ge 0.
\end{align}
\end{definition}

\noindent
If $g \in \svx(m,L)$ with $m>0$, then $x_\star$ is also the unique stationary point of $g$. It is worth noting that $\cvx(m,L) \subset \svx(m,L)$. In general, a function $g\in\svx(m,L)$ may not be convex. If $g\in \mathcal{S}(m, \infty)$, then $g$ may not be smooth. The condition \eqref{eq:gradientNL} is similar to the notion of \textit{one-point convexity}~\cite{arora2015, chen2015, sun2016guaranteed} and \textit{star-convexity}~\cite{lee2016optimizing}.

\subsection{Assumptions}

Referring to the problem setup~\eqref{eq:opt}, we will adopt the general assumption that $g\in\svx(m,\infty)$ with $m > 0$. So in general, $g$ may not be convex. We will analyze four different cases, characterized by different assumptions on individual $f_i$:
(\rom{1}) Bounded shifted gradients:\footnote{This case is a variant of the common assumption  $\frac{1}{n}\sum_{i=1}^n\norm{\nabla f_i(x)}^2 \le \beta$. One can check that this case holds for several $\ell_2$-regularized problems including SVM and logistic regression.} $\norm{\nabla f_i(x) - mx} \le \beta$ for all $x\in\R^p$; (\rom{2}) $f_i$ is $L$-smooth; (\rom{3})
 $f_i\in\cvx(0,L)$; 
(\rom{4}) $f_i\in\cvx(m,L)$. 

Assumption \rom{1}  is a natural assumption for SVM\footnote{The loss functions for SVM are non-smooth, and $u_k$ is actually updated using the subgradient information. For simplicity, we abuse our notation and use $\nabla f_i$ to denote the subgradient of $f_i$ for SVM problems.} and logistic regression while Assumptions \rom{2}, \rom{3}, or \rom{4} can be used for ridge regression, logistic regression, and smooth SVM. The $m$ assumed in cases~\rom{1} and \rom{4} is the same as the $m$ used in the assumption on $g\in\mathcal{S}(m,\infty)$.

\section{Analysis framework}
\label{sec:generalAnalysis}

\subsection{An LMI condition for the analysis of biased SGD}
 To analyze the convergence properties of biased SGD, we present a small linear matrix inequality (LMI) whose feasible points directly lead to convergence bounds of the biased SGD \eqref{eq:SGerror}
with the error model \eqref{eq:def_ykhk0}.

\begin{table}[!h]
\centering
\begin{tabular}{c|p{50mm}|l|l}\hline\rule{0pt}{2.6ex}%
Case & Desired assumption on the $f_i$ & Value of $M$ & Value of $G^2$ \\\hline\rule{0pt}{6.2mm}%
\rom{1} & $(f_i(x)-\frac{m}{2}\|x\|^2)$ have  bounded gradients; $\|\nabla f_i(x)-mx\|\le \beta$. & $\bmtx -m^2 & m \\ m & -1 \emtx$ & $\beta^2 + m^2\norm{x_\star}^2$\\[3mm]
\rom{2} & The $f_i$ are $L$-smooth, but are not necessarily convex. & $\bmtx 2L^2 & 0 \\ 0 & -1 \emtx$ & $\frac{1}{n}\sum_{i=1}^n \|\nabla f_i(x_\star)\|^2$\\[3mm]
\rom{3} & The $f_i$ are convex and \mbox{$L$-smooth}; $f_i \in \mathcal{F}(0,L)$. & $\bmtx 0 & L \\ L & -1 \emtx$ & $\frac{1}{n}\sum_{i=1}^n \|\nabla f_i(x_\star)\|^2$\\[3mm]
\rom{4} & The $f_i$ are $m$-strongly convex and $L$-smooth; $f_i \in \mathcal{F}(m,L)$ & $\bmtx -2mL & L+m \\ L+m & -1 \emtx$ & $ \frac{1}{n}\sum_{i=1}^n \|\nabla f_i(x_\star)\|^2$\\[3mm]
\hline
\end{tabular}
\caption{Given that $g\in \mathcal{S}(m, \infty)$, this table shows different possible assumptions about the $f_i$ and their corresponding values of $M$ and $G^2$ that will be used for our analysis.
\label{tab:Mmat}}
\end{table}

\begin{theorem}[Main Theorem]\label{thm:main}
Consider biased SGD~\eqref{eq:SGerror} with $g\in\svx(m,\infty)$ for some $m>0$, and let $x_\star$ be the unique global minimizer of $g$. Given one of the four conditions on $f_i$ and the corresponding $M=\bsmtx M_{11} & M_{12}\\ M_{21} & M_{22}\esmtx$ and $G$ from Table \ref{tab:Mmat},
if the following holds for some choice of nonnegative $\lambda_k, \nu_k, \mu_k$,$\rho_k$,
\begin{align}\label{eq:LMI_big}
\bmtx -\rho_k^2 - 2\nu_k m + \lambda_k M_{11}  & \nu_k + \lambda_k M_{12} & -1 & 0 \\
\nu_k +\lambda_k M_{21} & \mu_k \delta^2 +\lambda_k M_{22} & \alpha_k & 0 \\
-1 & \alpha_k &  -1 & \alpha_k \\
0 & 0 & \alpha_k & -\mu_k \emtx
\preceq 0 
\end{align}
where the inequality is taken in the \emph{semidefinite} sense, then the biased SGD iterates satisfy
\begin{align}\label{eq:bound}
\ee \|x_{k+1}-x_\star\|^2 \le \rho_k^{2}\ee \|x_k-x_\star\|^2  + (2\lambda_k G^2+\mu_k c^2)
\end{align}
\end{theorem}
\begin{proof}
The proof is based on extending the quadratic constraint approach in \cite{Lessard2014} to the stochastic case. Specifically, one can show that for each of the four conditions on $f_i$ and the corresponding $M$ and $G$ in Table \ref{tab:Mmat}, the following quadratic constraint holds.
\begin{align}
\label{eq:SQC3}
\ee\bmtx x_k-x_\star \\ u_k \emtx^\tp (M\otimes I_p)  \bmtx x_k-x_\star \\ u_k \emtx \ge -2G^2.
\end{align}
Then one can use  some standard arguments from the controls literature  to prove the statement in this theorem. A detailed proof is presented in the appendix.
\end{proof}

\begin{remark}
Under mild technical assumptions, the result in Theorem 1 can be extended for the problem in the more general form of $\min_x\{\mathbb{E}f_i(x)\}$, since its proof does not depend on the cardinality of the index set that $i$ is sampled from. For simplicity, our paper focuses on the finite sum setup.
\end{remark}

Notice \eqref{eq:bound} can be used to prove various types of convergence results. We will briefly discuss this in Remark \ref{rem:remark1} and provide more details in later sections.
 For a fixed $\delta$, the matrix in~\eqref{eq:LMI_big} is linear in $(\rho_k^2,\nu_k,\mu_k,\lambda_k,\alpha_k)$,  so~\eqref{eq:LMI_big} is a \emph{linear matrix inequality} (LMI) whose feasible set is convex and can be efficiently searched using standard semidefinite program solvers.
For example, one can implement the LMIs using CVX, a package for specifying and solving convex programs~\cite{cvx2,cvx1}. Since the matrix in \eqref{eq:LMI_big} is even linear in $\alpha_k$,   so the LMI \eqref{eq:LMI_big} can be used to study the impacts of adaptive stepsize rules on the performance of biased SGD from a theoretical viewpoint.
One may also obtain analytical formulas for certain feasibility points of the LMI~\eqref{eq:LMI_big} due to its simple form.
Our analytical bounds for biased SGD are based on the following result.

\begin{corollary}
\label{thm:cor1}
Choose one of the four conditions on $f_i$ and the corresponding $M=\bsmtx M_{11} & M_{12}\\ M_{21} & M_{22}\esmtx$ and $G$ from Table \ref{tab:Mmat}. Also define $\tilde M = M_{11} + 2mM_{12}$. Consider biased SGD~\eqref{eq:SGerror} with $g\in\svx(m,\infty)$ for some $m>0$, and let $x_\star$ be the unique global minimizer of $g$. Suppose the stepsize satisfies the bound $0<M_{21} \alpha_k\le 1$ \footnote{Ensuring such a condition in practice can be challenging for many cases since it heavily relies on the estimations of problem parameters.}(which is equivalent to the following upper bound on $\alpha_k$ for the four cases being considered in this paper).  
\begin{table}[!h]
	\centering
	\begin{tabular}{c|c|c|c|c} \hline\rule{0pt}{2.6ex}
		Case  & \rom{1} & \rom{2} & \rom{3} & \rom{4} \\\hline\rule{0pt}{2.6ex}
		$M_{21}$  & $m$ & $0$ & $L$ & $L+m$  \\\rule{0pt}{2.6ex}
		$\tilde M = M_{11} + 2mM_{12}$  & $m^2$ & $2L^2$ & $2mL$ & $2m^2$ \\\rule{0pt}{2.6ex}
		$\alpha_k$ bound & $\frac{1}{m}$ & $\infty$ & $\frac{1}{L}$ & $\frac{1}{L+m}$ \\\hline
	\end{tabular}
\end{table}

Then biased SGD \eqref{eq:SGerror} with the error model~\eqref{eq:def_ykhk0} satisfies the bound \eqref{eq:bound} with the following nonnegative parameters
\begin{subequations}\label{eq:mainFor}
\begin{align}
\mu_k&=\alpha_k^2(1+\zeta_k^{-1})\\
\lambda_k&=\alpha_k^2(1+\zeta_k)(1+\delta^2\zeta_k^{-1})\\
\rho_k^2&=(1+\zeta_k)(1-2m\alpha_k+\tilde M\alpha_k^2 (1+\delta^2\zeta_k^{-1}))\label{bad}
\end{align}
\end{subequations}
where $\zeta_k$ is a parameter that satisfies $\zeta_k>0$ and $\zeta_k \ge \frac{\alpha_k M_{21}\delta^2}{1-\alpha_k M_{21}}$. Each choice of~$\zeta_k$ yields a different bound in~\eqref{eq:bound}.
\end{corollary}
\begin{proof}
We further define 
\begin{align}
\label{eq:mainFor1}
\nu_k=\alpha_k (1+\zeta_k)(1-\alpha_k M_{21}(1+\delta^2 \zeta_k^{-1}))
\end{align}
We will show that \eqref{eq:mainFor} and \eqref{eq:mainFor1} are a feasible solution for \eqref{eq:LMI_big}.
We begin with~\eqref{eq:LMI_big} and take the Schur complement with respect to the $(3,3)$ entry of the matrix, leading to
\begin{align}\label{eq:threebythree}
\bmtx 1-\rho_k^2 - 2\nu_k m + \lambda_k M_{11}   & \nu_k + \lambda_k M_{12} - \alpha_k             & -\alpha_k \\
\nu_k + \lambda_k M_{21} - \alpha_k              & \mu_k \delta^2 + \lambda_k M_{22} + \alpha_k^2  & \alpha_k^2 \\
-\alpha_k                                      & \alpha_k^2                                    & \alpha_k^2-\mu_k \emtx \preceq 0
\end{align}
Examining the $(3,3)$ entry, we deduce that $\mu_k > \alpha_k^2$, for if we had equality instead, the rest of the third row and column would be zero, forcing $\alpha_k=0$. Substituting $\mu_k = \alpha_k^2(1 + \zeta_k^{-1})$ for some $\zeta_k > 0$ and taking the Schur complement with respect to the $(3,3)$ entry, we see~\eqref{eq:threebythree} is equivalent to
\begin{align}\label{eq:twobytwotwo}
\bmtx 1-\rho_k^2 - 2\nu_k m + \lambda_k M_{11} + \zeta_k & \nu_k + \lambda_k M_{12} - \alpha_k(1+\zeta_k)\\
\nu_k + \lambda_k M_{21} - \alpha_k(1+\zeta_k) & \lambda_k M_{22} + \alpha_k^2(1+\zeta_k)(1+\delta^2 \zeta_k^{-1}) \emtx \preceq 0
\end{align}
In~\eqref{eq:twobytwotwo}, $\zeta_k > 0$ is a parameter that we are free to choose, and each choice yields a different set of feasible tuples $(\rho_k^2,\lambda_k, \mu_k, \nu_k)$. One way to obtain a feasible tuple is to set the left side of \eqref{eq:twobytwotwo} equal to the zero matrix. This shows \eqref{eq:LMI_big} is feasible with the following parameter choices.
\begin{subequations}
\label{eq:mainFor3}
\begin{align}
\mu_k&=\alpha_k^2(1+\zeta_k^{-1})\\
\lambda_k&=-\alpha_k^2(1+\zeta_k)(1+\delta^2 \zeta_k^{-1})M_{22}^{-1}\\
\nu_k&=\alpha_k(1+\zeta_k)-\lambda_k M_{21}\\
\rho_k^2&=1-2\nu_k m+\lambda_k M_{11} +\zeta_k
\end{align}
\end{subequations}
Since we always have $M_{22}=-1$ in Table~\ref{tab:Mmat}, it is straightforward to verify that \eqref{eq:mainFor3} is equivalent to \eqref{eq:mainFor} and \eqref{eq:mainFor1}. Notice that we directly have $\mu_k\ge 0$ and $\lambda_k\ge 0$ because $\zeta_k>0$.
In order to ensure $\rho_k^2 \ge 0$ and $\nu_k\ge 0$, we must have
$1-2m\alpha_k+\tilde M\alpha_k^2 (1+\delta^2\zeta_k^{-1}) \ge 0$
and $\alpha_k M_{21} (1+\delta^2\zeta_k^{-1})\le 1$, respectively.  The first inequality always holds because $\tilde M \ge m^2$ and we have $1-2m\alpha_k+\tilde M\alpha_k^2(1+\delta^2\zeta_k^{-1}) \ge 1-2m\alpha_k+m^2\alpha_k^2 \ge (1-m\alpha_k)^2 \ge 0$. Based on the conditions $0 \le \alpha_k M_{21} < 1$ and $\zeta_k \ge \frac{\alpha_k M_{21}\delta^2}{1-\alpha_k M_{21}}$, we conclude that the second inequality always holds as well. Since we have constructed a feasible solution to the LMI~\eqref{eq:LMI_big}, the bound~\eqref{eq:bound} follows from Theorem \ref{thm:main}.
\end{proof}

Given $\alpha_k$, Corollary \ref{thm:cor1} provides a one-dimensional family of solutions to the LMI \eqref{eq:LMI_big}. 
These solutions are given by \eqref{eq:mainFor} and \eqref{eq:mainFor1} and are parameterized by the auxiliary variable $\zeta_k$.
 Corollary \ref{thm:cor1} does not require $\rho_k\le 1$. Hence it actually does not impose any upper bound on $\alpha_k$ in Case \rom{2}. Later we will impose refined upper bounds on $\alpha_k$ such that the bound \eqref{eq:bound} can be transformed into a useful bound in the form of \eqref{eq:constMainBound}. We also want to mention that the stepsize bounds in the above corollary are consistent with the existing results in the machine learning literature. For example, for Case \rom{3}, the stepsize bound for the standard SGD method is known to be $1/L$ (see Theorem 2.1 in \cite{needell2014}).

\begin{remark}
\label{rem:remark1}
We can use \eqref{eq:bound} to obtain  various types of convergence results.
 For example, when a constant stepsize is used, i.e. $\alpha_k=\alpha$ for all $k$, a naive analysis can be performed by setting $\zeta_k=\zeta$ for all $k$. In this case,
$(\rho_k, \nu_k, \mu_k, \lambda_k)$ are set to be constants $(\rho, \nu, \mu, \lambda)$.
Then, \eqref{eq:LMI_big} and \eqref{eq:bound} become independent of $k$. We can rewrite \eqref{eq:bound} as
\begin{align}
\label{eq:keyineq2}
\ee \|x_{k+1}-x_\star\|^2\le \rho^2\ee  \|x_k-x_\star\|^2 +   (2\lambda G^2+\mu c^2).
\end{align}
If $\rho < 1$, then we may recurse \eqref{eq:keyineq2} to obtain the following convergence result:
\begin{align}
\label{eq:keyineq3}
\ee \|x_{k}-x_\star\|^2&\le \rho^{2k} \ee \|x_0-x_\star\|^2 + \left(\sum_{i=0}^{k-1}\rho^{2i}\right) \left( 2\lambda G^2+\mu c^2\right)\notag\\
&\le \rho^{2k} \ee \|x_0-x_\star\|^2+\frac{2\lambda G^2+\mu c^2}{1-\rho^2}.
\end{align}
The inequality~\eqref{eq:keyineq3} is an error bound of the familiar form~\eqref{eq:constMainBound}. Nevertheless, this bound may be conservative even in the constant stepsize case. To minimize the right-hand side of \eqref{eq:bound},  the objective function for the semidefinite program \eqref{eq:LMI_big} at step $k$ should be chosen as $\rho_k^{2}\ee \|x_k-x_\star\|^2  + (2\lambda_k G^2+\mu_k c^2)$. Consequently, setting $\zeta_k$ to be a constant may introduce conservatism even in the constant stepsize case. To overcome this issue, we will introduce a sequential minimization approach next.
\end{remark}

\subsection{Sequential minimization approach for biased SGD}
\label{sec:sequential}
We will quantify the convergence behaviors of biased SGD by providing upper bounds for $\ee \|x_k-x_\star\|^2$. To do so, we will recursively make use of the bound~\eqref{eq:bound}. Suppose $\delta$, $c$, and $G$ are constant. Define $\mathcal{T}_k \subseteq \R_+^4$ to be the set of tuples $(\rho_k,\lambda_k,\mu_k,\nu_k)$ that are feasible points for the LMI~\eqref{eq:LMI_big}.
Also define the real number sequence $\{U_k\}_{k\ge 0}$ via the recursion:
\begin{align}\label{eq:recursion}
U_0 \ge \ee \norm{x_0-x_\star}^2
\qquad\text{and}\qquad
U_{k+1} = \rho_k^2 U_k + 2\lambda_k G^2 + \mu_k c^2
\end{align}
where $(\rho_k,\lambda_k,\mu_k,\nu_k) \in \mathcal{T}_k$. By induction, we can show that $U_k$ provides an upper bound for the error at timestep $k$. Indeed, if $\ee \norm{x_k-x_\star}^2 \le U_k$, then by Theorem~\ref{thm:main}, we have
$\ee\norm{x_{k+1}-x_\star}^2
\le \rho_k^{2}\ee \|x_k-x_\star\|^2  + 2\lambda_k G^2+\mu_k c^2 
\le \rho_k^{2}U_k + 2\lambda_k G^2+\mu_k c^2
= U_{k+1}$.
A key issue in computing a useful upper bound $U_k$ is how to choose the tuple $(\rho_k,\lambda_k,\mu_k,\nu_k)\in\mathcal{T}_k$.
If the stepsize is constant ($\alpha_k = \alpha$), then $\mathcal{T}_k$ is independent of $k$. Thus we may choose the same particular solution $(\rho,\lambda,\mu,\nu)$ for each $k$. Then, based on \eqref{eq:keyineq3}, if $\rho < 1$ we can obtain a bound of the following form for biased SGD:
\begin{align}
\label{eq:constBound}
\ee \|x_k-x_\star\|^2\le \rho^{2k} U_0+\frac{2\lambda G^2+\mu c^2}{1-\rho^2}.
\end{align} 
As discussed in Remark \ref{rem:remark1}, the above bound may be unnecessarily conservative.
Because of the recursive definition~\eqref{eq:recursion}, the bound $U_k$ depends solely on $U_0$ and the parameters $\{\rho_t,\lambda_t,\mu_t\}_{t=0}^{k-1}$. So we can seek the smallest possible upper bound by solving the optimization problem:
\[
\begin{aligned}
U_{k+1}^\textup{opt}\,=\,
\minimize_{\{\rho_t,\lambda_t,\mu_t,\nu_t\}_{t=0}^{k}} \qquad & U_{k+1} \\
\text{subject to} \qquad & U_{t+1} = \rho_t^2 U_t + 2\lambda_t G^2 + \mu_t c^2 & 0\le t \le k \\
& (\rho_t,\lambda_t,\mu_t,\nu_t) \in \mathcal{T}_t& 0\le t \le k
\end{aligned}
\]
A useful fact is that the above optimization problem can be solved in a sequential manner. Formally, we have
\begin{proposition}
The following holds for all $k$.
\begin{align}
U_{k+1}^\textup{opt}=\minimize_{(\rho,\lambda,\mu,\nu)\in\mathcal{T}_k}\quad
\rho^2 U_k^\textup{opt} + 2\lambda G^2 + \mu c^2
\end{align}
Consequently, a greedy approach where $U_{t+1}$ is optimized in terms of $U_t$ recursively for $t=0,\dots,k-1$ yields a bound $U_k$ that is in fact globally optimal over all possible choices of parameters $\{\rho_t,\lambda_t,\mu_t,\nu_t\}_{t=0}^{k}$.
\end{proposition}
\begin{proof}
This optimization problem being considered is similar to a dynamic programming and a recursive solution reminiscent of the Bellman equation can be derived for the optimal bound $U_k$.
\begin{align}\label{yuck}\notag
&U_{k+1}^\textup{opt} \\
&=
\minimize_{(\rho,\lambda,\mu,\nu)\in\mathcal{T}_k}
\left\{
\begin{aligned}
\minimize_{\{\rho_t,\lambda_t,\mu_t,\nu_t\}_{t=0}^{k-1}} \quad & U_{k+1} \\\notag
\text{subject to} \quad & U_{t+1} = \rho_t^2 U_t + 2\lambda_t G^2 + \mu_t c^2 && 0\le t \le k \\
& (\rho_t,\lambda_t,\mu_t,\nu_t) \in \mathcal{T}_t&& 0\le t < k \\
& (\rho,\lambda,\mu,\nu) = (\rho_k,\lambda_k,\mu_k,\nu_k)
\end{aligned} \right\} \\[3mm]
&=
\minimize_{(\rho,\lambda,\mu,\nu)\in\mathcal{T}_k}
\left\{
\begin{aligned}
\minimize_{\{\rho_t,\lambda_t,\mu_t,\nu_t\}_{t=0}^{k-1}} \quad & \rho^2 U_k + 2\lambda G^2 + \mu c^2 \\ \notag
\text{subject to} \quad & U_{t+1} = \rho_t^2 U_t + 2\lambda_t G^2 + \mu_t c^2 && 0 \le t < k \\
& (\rho_t,\lambda_t,\mu_t,\nu_t) \in \mathcal{T}_t&& 0\le t < k
\end{aligned} \right\} \\[3mm]
&= \minimize_{(\rho,\lambda,\mu,\nu)\in\mathcal{T}_k}\quad
\rho^2 U_k^\textup{opt} + 2\lambda G^2 + \mu c^2
\end{align}
Where the final equality in~\eqref{yuck} relies on the fact that $\rho^2 \ge 0$.
\end{proof}

Obtaining an explicit analytical formula for $U_k^\textup{opt}$ is not straightforward, since it involves solving a sequence of semidefinite programs. However, we can make use of Corollary~\ref{thm:cor1} to further upper-bound $U_k^\textup{opt}$. This works because Corollary~\ref{thm:cor1} gives an analytical parameterization of a \textit{subset} of $\mathcal{T}_k$. Denote this new upper bound by $\hat U_k$. By Corollary~\ref{thm:cor1}, we have:
\begin{equation}\label{eq:opt_simpler}
\begin{aligned}
\hat U_{k+1} \,\,=\,\, \minimize_{\zeta > 0} \qquad &
\rho^2 \hat U_k + 2\lambda G^2 + \mu c^2 \\
\text{subject to}\qquad &
\mu =\alpha_k^2(1+\zeta^{-1})\\
&\lambda =\alpha_k^2(1+\zeta)(1+\delta^2\zeta^{-1})\\
&\rho^2 =(1+\zeta)(1-2m\alpha_k+\tilde M\alpha_k^2 (1+\delta^2\zeta^{-1})) \\
&\zeta \ge \tfrac{\alpha_k M_{21}\delta^2}{1-\alpha_k M_{21}}
\end{aligned}
\end{equation}
Note that Corollary~\ref{thm:cor1} also places bounds on $\alpha_k$, which we assume are being satisfied here.
The optimization problem~\eqref{eq:opt_simpler} is a single-variable smooth constrained problem. It is straightforward to verify that $\mu$, $\lambda$, and $\rho^2$ are convex functions of $\zeta$ when $\zeta > 0$. Moreover, the inequality constraint on $\zeta$ is linear, so we deduce that~\eqref{eq:opt_simpler} is a convex optimization problem.

Thus, we have reduced the problem of recursively solving semidefinite programs (finding $U_k^\text{opt}$) to recursively solving single-variable convex optimization problems (finding $\hat U_k$). Ultimately, we obtain an upper bound on the expected error of biased SGD that is easy to compute:
\begin{align}\label{eq:bound_approx}
\ee \norm{x_k-x_\star}^2 \le U_k^\textup{opt} \le \hat U_k
\end{align}
Preliminary numerical simulations suggest that $\hat U_k$ seems to be equal to $U_k^\textup{opt}$ under the four sets of  assumptions in this paper.
However, we are unable to show $\hat{U_k}=U_k^\textup{opt}$ analytically. In the subsequent sections, we will solve the recursion for $\hat U_k$ analytically and derive convergence bounds for biased SGD.

\subsection{Analytical recursive bounds for biased SGD}
We showed in the previous section that $\ee \|x_k-x_\star\|^2\le \hat U_k$ for biased SGD, where $\hat U_k$ is the solution to~\eqref{eq:opt_simpler}. We now derive an analytical recursive formula for $\hat U_k$.
Let us simplify the optimization problem~\eqref{eq:opt_simpler}. Eliminating $\rho,\lambda,\mu$, we obtain
\begin{equation}\label{eq:opt_simpler2}
\begin{aligned}
\hat U_{k+1} \,\,=\,\, \minimize_{\zeta > 0} \qquad &
a_k(1+\zeta^{-1}) + b_k(1+\zeta)\\
\text{subject to}\qquad
&a_k=\alpha_k^2 \left( c^2+2\delta^2 G^2+\tilde M\delta^2 \hat U_k\right)\\
&b_k=\left(1-2m\alpha_k+\tilde M\alpha_k^2\right)\hat U_k+2\alpha_k^2 G^2\\
&\zeta \ge \tfrac{\alpha_k M_{21}\delta^2}{1-\alpha_k M_{21}}\\
\end{aligned}
\end{equation}
The assumptions on $\alpha_k$ from Corollary~\ref{thm:cor1} also imply that $a_k \ge 0$ and $b_k \ge 0$. We may now solve this problem explicitly and we summarize the solution to~\eqref{eq:opt_simpler2} in the following lemma.

\begin{lemma}
\label{lem:recursiveBound}
Consider biased SGD~\eqref{eq:SGerror} with $g\in\svx(m,\infty)$ for some $m>0$, and let $x_\star$ be the unique global minimizer of $g$. Given one of the four conditions on $f_i$ and the corresponding $M=\bsmtx M_{11} & M_{12}\\ M_{21} & M_{22}\esmtx$ and $G$ from Table \ref{tab:Mmat}, further assume $\alpha_k$ is strictly positive and satisfies $M_{21} \alpha_k\le 1$.
Then the error bound $\hat U_k$ defined in~\eqref{eq:opt_simpler2} can be computed recursively as follows. 
\begin{align}\label{eq:recursion_Uk_constrained}
\hat U_{k+1} &=
\begin{cases}
\left( \sqrt{a_k} + \sqrt{b_k} \right)^2 & \sqrt{\frac{a_k}{b_k}} \ge \frac{\alpha_k M_{21} \delta^2}{1-\alpha_k M_{21}} \\
a_k+b_k+ a_k \frac{1-\alpha_k M_{21}}{\alpha_k M_{21} \delta^2}+b_k \frac{\alpha_k M_{21} \delta^2}{1-\alpha_kM_{21}} & \text{otherwise}
\end{cases}
\end{align}
where $a_k$ and $b_k$ are defined in~\eqref{eq:opt_simpler2}. We may initialize the recursion at any $\hat U_0 \ge \ee \norm{x_0-x_\star}^2$.
\end{lemma}
\begin{proof}
In Case \rom{2}, we have $M_{21}=0$ so the constraint on $\zeta$ is vacuously true.  Therefore, the only constraint on $\zeta$ in~\eqref{eq:opt_simpler2} is $\zeta>0$ and we can solve the problem by setting the derivative of the objective function with respect to $\zeta$ equal to zero. The result is $\zeta_k =\sqrt{\frac{a_k}{b_k}}$. In Cases \rom{1}, \rom{3}, and \rom{4}, we have $M_{21}>0$. By convexity, the optimal $\zeta_k$ is either the unconstrained optimum (if it is feasible) or the boundary point (otherwise). Hence \eqref{eq:recursion_Uk_constrained} holds as desired.
Note that if $\delta=c=0$, then $a_k=0$. This corresponds to the pathological case where the objective reduces to $b_k(1+\zeta)$. Here, the optimum is achieved as $\zeta\to 0$, which corresponds to $\mu\to\infty$ in~\eqref{eq:opt_simpler}. This does not cause a problem because $c=0$ so $\mu$ does not appear in the objective function. The recursion~\eqref{eq:recursion_Uk_constrained} then simplifies to $\hat U_{k+1} = b_k$.
\end{proof}

\begin{remark}\label{rem:unconstrained}
If $M_{21}=0$ (Case~\rom{2} in Table~\ref{tab:Mmat}) or if $\delta=0$ (no multiplicative noise), the optimization problem~\eqref{eq:opt_simpler2} reduces to an unconstrained optimization problem whose solution is
\begin{align}\label{eq:recursion_Uk_unconstrained}\notag
\hat U_{k+1} &= \left( \sqrt{a_k} + \sqrt{b_k} \right)^2 \\
&= \left( \alpha_k \sqrt{c^2+2\delta^2 G^2+\tilde M\delta^2 \hat U_k} + \sqrt{\left(1-2m\alpha_k+\tilde M\alpha_k^2\right)\hat U_k+2G^2\alpha_k^2} \right)^2
\end{align}
\end{remark}

\section{Analytical rate bounds for the constant stepsize case}
\label{sec:fixed_alpha}

In this section, we present non-recursive error bounds for biased SGD with constant stepsize. Specifically, we assume $\alpha_k=\alpha$ for all $k$ and we either apply Lemma~\ref{lem:recursiveBound} or carefully choose a constant $\zeta$ in order to obtain a tractable bound for $\hat U_k$. The bounds derived in this section highlight the trade-offs inherent in the design of biased SGD. 

\subsection{Linearization of the nonlinear recursion}
\label{sec:linearization}

This first result applies to the case where $\delta=0$ or $M_{21}=0$ (Case \rom{2}) and leverages Remark~\ref{rem:unconstrained} to obtain a bound for biased SGD.
\begin{corollary}\label{cor:linear1}
Consider biased SGD~\eqref{eq:SGerror} with $g\in\svx(m,\infty)$ for some $m>0$, and let $x_\star$ be the unique global minimizer of $g$. Given one of the four conditions on $f_i$ and the corresponding $M=\bsmtx M_{11} & M_{12}\\ M_{21} & M_{22}\esmtx$ and $G$ from Table \ref{tab:Mmat}, further assume that $\alpha_k=\alpha>0$ (constant stepsize), $M_{21}\alpha \le 1$, and either $\delta=0$ or $M_{21}=0$. Define $p,q,r,s \ge 0$ as follows.
\begin{align}\label{eq:sub_values_pqrs}
p &= \tilde M \delta^2 \alpha^2,
&
q &= (c^2 + 2G^2\delta^2)\alpha^2,
&
r &= 1-2m \alpha +\tilde M \alpha^2,
&
s &= 2 G^2\alpha ^2.
\end{align}
Where $\tilde M = M_{11} + 2mM_{12}$. If $\sqrt{p}+\sqrt{r} < 1$ then we have the following iterate error bound:
\begin{align}\label{eq:bound_nonrecursive}
\ee \norm{x_k-x_\star}^2 \le 
\left( \tfrac{p\sqrt{\hat U_\star}}{\sqrt{p \hat U_\star + q}} + \tfrac{r\sqrt{\hat U_\star}}{\sqrt{r \hat U_\star + s}} \right)^k \ee \norm{x_0-x_\star}^2 + \hat U_\star,
\end{align}
where the fixed point $\hat U_\star$ is given by
\begin{equation}\label{eq:Ustar_pqrs}
\hat U_\star = \frac{(p-r) (s-q)+q+s + 2 \sqrt{p s^2+q^2 r+q s (1-p-r)}}{(p-r)^2-2 (p+r)+1}.
\end{equation}
\end{corollary}
\begin{proof}
By Remark~\ref{rem:unconstrained}, we have the nonlinear recursion~\eqref{eq:recursion_Uk_unconstrained} for $\hat U_k$. This recursion is of the form
\begin{equation}\label{eq:recursion_pqrs}
\hat U_{k+1} = \left( \sqrt{p \hat U_k + q} + \sqrt{r \hat U_k + s} \right)^2,
\end{equation}
where $p,q,r,s > 0$ are given in~\eqref{eq:sub_values_pqrs}.
It is straightforward to verify that the right-hand side of~\eqref{eq:recursion_pqrs} is a monotonically increasing concave function of $\hat U_k$ and its asymptote is a line of slope $(\sqrt{p}+\sqrt{r})^2$. Thus,~\eqref{eq:recursion_pqrs} will have a unique fixed point when $\sqrt{p}+\sqrt{r} < 1$. We will return to this condition shortly. When a fixed point exists, it is found by setting $\hat U_k = \hat U_{k+1} = \hat U_\star$ in~\eqref{eq:recursion_pqrs} and yields~$U_\star$ given by~\eqref{eq:Ustar_pqrs}.
The concavity property further guarantees that any first-order Taylor expansion of the right-hand side of~\eqref{eq:recursion_pqrs} yields an upper bound to $\hat U_{k+1}$. Expanding about $\hat U_\star$, we obtain:
\begin{equation}\label{eq:delta_noisy}
\hat U_{k+1}-\hat U_\star \le \left( \tfrac{p\sqrt{\hat U_\star}}{\sqrt{p \hat U_\star + q}} + \tfrac{r\sqrt{\hat U_\star}}{\sqrt{r \hat U_\star + s}} \right) \left(\hat U_k - \hat U_\star\right)
\end{equation}
which leads to the following non-recursive bound for biased SGD.
\begin{align}\label{eq:bound_nonrecursive1}
\notag
\ee \norm{x_k-x_\star}^2 \le
\hat U_{k} &\le  \left( \tfrac{p\sqrt{\hat U_\star}}{\sqrt{p \hat U_\star + q}} + \tfrac{r\sqrt{\hat U_\star}}{\sqrt{r \hat U_\star + s}} \right)^k (\hat U_0 -\hat U_\star) + \hat U_\star\\
&\le \left( \tfrac{p\sqrt{\hat U_\star}}{\sqrt{p \hat U_\star + q}} + \tfrac{r\sqrt{\hat U_\star}}{\sqrt{r \hat U_\star + s}} \right)^k \hat U_0  + \hat U_\star
\end{align}
Since this bound holds for any $\hat U_0 \ge \ee \norm{x_0-x_\star}^2$, it holds in particular when we have equality, and thus we obtain~\eqref{eq:bound_nonrecursive} as required.
\end{proof}

The condition that $\sqrt{p}+\sqrt{r}<1$ from Corollary~\ref{cor:linear1}, which is necessary for the existence of a fixed-point of~\eqref{eq:recursion_pqrs}, is equivalent to an upper bound on~$\alpha$. After manipulation, it amounts to:
\begin{equation}\label{alph2}
\alpha < \frac{2(m-\delta\sqrt{\tilde M})}{\tilde M(1-\delta^2)}
\end{equation}
Therefore, we can ensure that $\sqrt{p}+\sqrt{r}<1$ when $\delta<m/\sqrt{\tilde M}$, and $\alpha$ is sufficiently small. 
If $\delta=0$, 
the stepsize bound \eqref{alph2} is only relevant in Case \rom{2}. For Cases \rom{1}, \rom{3}, and \rom{4}, the bound $M_{21}\alpha\le 1$ imposes a stronger restriction on $\alpha$ (see Corollary \ref{thm:cor1}).  If $\delta \neq 0$, we only consider Case \rom{2} ($M_{21}=0$) and the resultant bound for $\alpha$ is $\frac{m-\sqrt{2}L\delta}{L^2(1-\delta^2)}$. 
The condition $\delta<m/\sqrt{\tilde M}$ becomes $\delta<m/(\sqrt{2}L)$.

To see the trade-offs in the design of biased SGD, we can take Taylor expansions of several key quantities about $\alpha=0$ to see how changes in $\alpha$ affect convergence:
\begin{subequations}\label{eq:asymalpha1}
	\begin{gather} 
	\label{e:ustar}
	\hat U_\star \approx \frac{c^2+2 \delta ^2 G^2}{m^2-\delta ^2 \tilde M} +
	\frac{m \left(c^2 (\tilde M-m^2)+2 \left(1-\delta ^2\right) G^2 m^2\right)}{(m^2-\delta ^2 \tilde M)^2}\alpha + O(\alpha^2)\\\label{rhorate1}
	\left( \tfrac{p\sqrt{\hat U_\star}}{\sqrt{p \hat U_\star + q}} + \tfrac{r\sqrt{\hat U_\star}}{\sqrt{r \hat U_\star + s}} \right) \approx 1-\frac{(m^2-\delta ^2\tilde M)}{m}\alpha + O(\alpha^2)
	\end{gather}
\end{subequations}
We conclude that when $\delta<m/\sqrt{\tilde M}$,  biased SGD converges linearly to a ball whose radius is roughly $\hat U_\star \ge 0$.
One can decrease the stepsize $\alpha$ to control the final error $\hat U_\star$. However, due to the errors in the individual gradient updates, one cannot guarantee the final error $\ee\norm{x_k-x_\star}^2$ smaller than
$\frac{c^2+2 \delta ^2 G^2}{m^2-\delta ^2 \tilde M}$. This is consistent with our intuition; one could inject noise in an adversarial manner to shift the optimum point away from $x_\star$ so there is no way to guarantee that $\{x_k\}$ converges to $x_\star$ just by decreasing the stepsize $\alpha$.

\begin{remark}\label{rem:differentiability}
One can check that the left side of \eqref{rhorate1} is not differentiable at $(c, \alpha)=(0, 0)$. Consequently, taking a Taylor expansion with respect to $\alpha$ and then setting $c=0$ does not yield the same result as first setting $c=0$ and then taking a Taylor expansion with respect to $\alpha$ of the resulting expression. This explains why \eqref{rhorate1} does not reduce to $\rho^2=1-2m\alpha+O(\alpha^2)$ when $c=\delta=0$.
It is worth noting that the higher order term $O(\alpha^2)$ in \eqref{rhorate1} depends on $c$. Indeed, it blows up as $c\to 0$. Therefore, the rate formula \eqref{rhorate1} only describes the stepsize design trade-off for a fixed positive $c$ and sufficiently small $\alpha$. 
Similar situation even holds for the case where $G=0$. As long as $c\neq 0$, the rate formula is not going to reduce to $\rho^2=1-2m\alpha+O(\alpha^2)$  due to the fact that the left side of \eqref{rhorate1}  is not differentiable at $(c,\alpha)=(0,0)$.
\end{remark}

The non-recursive bound~\eqref{eq:bound_nonrecursive} relied on a linearization of the recursive formula \eqref{eq:recursion_pqrs}, which involved a time-varying $\zeta_k$. It is emphasized that we assumed that either $\delta=0$ or $M_{21}=0$.
In the other cases, namely $\delta >0$ and $M_{21}>0$ (Case \rom{1}, \rom{3}, or \rom{4}), we cannot ignore the additional condition $\zeta_k\ge \frac{\alpha M_{21} \delta^2}{1-\alpha M_{21}}$ and we must use the hybrid recursive formula~\eqref{eq:recursion_Uk_constrained}. This hybrid formulation is more problematic to solve explicitly. However, if we are mostly interested in the regime where $\alpha$ is small, we can obtain non-recursive bounds similar to~\eqref{eq:bound_nonrecursive} by carefully choosing a constant $\zeta$ for all $k$. We will develop these bounds in the next section.

\subsection{Non-recursive bounds via a fixed $\zeta$ parameter}

When $\alpha$ is small, we can choose $\zeta=m\alpha$ and we obtain the following result.

\begin{corollary}\label{cor:linear2}
Consider biased SGD~\eqref{eq:SGerror} with $g\in\svx(m,\infty)$ for some $m>0$, and let $x_\star$ be the unique global minimizer of $g$. Given one of the four conditions on $f_i$ and the corresponding $M=\bsmtx M_{11} & M_{12}\\ M_{21} & M_{22}\esmtx$ and $G$ from Table \ref{tab:Mmat}, further assume that $\alpha_k=\alpha>0$ (constant stepsize),  and $M_{21}\left(\alpha+\frac{\delta^2}{m}\right)\le 1$.\footnote{When $M_{21}=0$, this condition always holds. When $\delta=0$, this condition is equivalent to $M_{21}\alpha \le 1$. Hence the above corollary can be directly applied if $M_{21}=0$ or $\delta=0$. If $M_{21}> 0$ and $\delta>0$, the condition $M_{21}\left(\alpha+\frac{\delta^2}{m}\right)\le 1$ can be rewritten as a condition on $\alpha$ in a case-by-case manner.} Finally, assume that
\begin{equation}\label{rho_messy}
0 < \tilde \rho^2 < 1
\quad\text{where }
\tilde \rho^2 = 1-\tfrac{m^2-\tilde M \delta^2}{m}\alpha+(\tilde M(1+\delta^2)-2m^2) \alpha^2+\tilde M m\alpha^3.
\end{equation}
Note \eqref{rho_messy} holds for $\alpha$ sufficiently small.
Then, we have the following error bound for the iterates
\begin{align}
\label{eq:constF}
\ee \|x_k-x_\star\|^2 \le \tilde \rho^{2k}\ee \norm{x_0-x_\star}^2 +\tilde U_\star
\end{align}
where $\tilde U_\star$ is given by
\begin{align}\label{utildestar}
\tilde U_\star=\frac{2\delta^2 G^2+c^2+m(c^2+ 2G^2(1+\delta^2))\alpha+2m^2G^2\alpha^2}{(m^2-\tilde M \delta^2)-m(\tilde M(1+\delta^2)-2m^2) \alpha-\tilde M m^2\alpha^2}
\end{align}
\end{corollary}
\begin{proof}
Set $\zeta=m\alpha$ in the optimization problem~\eqref{eq:opt_simpler2}. This defines a new recursion for a quantity $\tilde U_k$ that upper-bounds $\hat U_k$ since we are choosing a possibly sub-optimal $\zeta$. Our assumption $M_{21}\left(\alpha+\frac{\delta^2}{m}\right)\le 1$ guarantees that $\zeta \ge \frac{\alpha M_{21}\delta^2}{1-\alpha M_{21}}$ when $\zeta = m\alpha$. Hence our choice of $\zeta$ is a feasible choice for \eqref{eq:opt_simpler2}. This leads to:
\begin{align*}
\tilde U_{k+1} &= a_k(1+\tfrac{1}{m\alpha}) + b_k(1+m\alpha) \\
&= \tilde \rho^2 \tilde U_k + \bigl(\alpha^2 ( c^2+2\delta^2 G^2)(1+\tfrac{1}{m\alpha}) +
2\alpha^2 G^2(1+m\alpha) \bigr)
\end{align*}
This is a simple linear recursion that we can solve explicitly in a similar way to the recursion in Remark~\ref{rem:remark1}. After simplifications, we obtain~\eqref{eq:constF} and \eqref{utildestar}.
\end{proof}

The linear rate of convergence in~\eqref{eq:constF} is of the same order as the one obtained in Corollary~\ref{cor:linear1} and~\eqref{rhorate1}. Namely,
\begin{align}
\label{eq:trho_approx}
\tilde \rho^2 \approx 1-\frac{(m^2-\tilde M \delta^2)}{m}\alpha+O(\alpha^2)
\end{align}
Likewise, the limiting error $\tilde U_\star$ from \eqref{utildestar} can be expanded as a series in $\alpha$ and we obtain a result that matches the small-$\alpha$ limit of $\hat U_\star$ from~\eqref{e:ustar} up to linear terms. Namely,
\begin{align}\label{utilde}
\tilde U_\star\approx \frac{c^2+2\delta^2 G^2}{m^2-\tilde M \delta^2}+
\frac{m \left(c^2 (\tilde M-m^2)+2 \left(1-\delta ^2\right) G^2 m^2\right)}{(m^2-\delta ^2 \tilde M)^2}\alpha + O(\alpha^2)
\end{align}
Therefore, \eqref{eq:constF} can give a reasonable non-recursive bound for biased SGD with small $\alpha$ even for the cases where $M_{21}>0$ and $\delta >0$.

Now we discuss the acceptable relative noise level under various assumptions on $f_i$. Based on \eqref{eq:trho_approx}, we need $m^2-\tilde M \delta^2>0$ to ensure $\tilde \rho^2<1$ for sufficiently small $\alpha$. The other constraint  $M_{21}\left(\alpha+\frac{\delta^2}{m}\right)\le 1$ enforces $M_{21} \delta^2<m$. 
Depending on which case we are dealing with, the conditions $\delta<m/\sqrt{\tilde M}$ and $M_{21} \delta^2 <m$ impose an upper bound on admissible values of $\delta$. See Table~\ref{tab:deltabound}.

\begin{table}[h!]
\centering
\begin{tabular}{c|c|c|c|c} \hline\rule{0pt}{2.6ex}
Case & \rom{1} & \rom{2} & \rom{3} & \rom{4}  \\\hline\rule{0pt}{2.6ex}
$\tilde M = M_{11} + 2mM_{12}$ & $m^2$ & $2L^2$ & $2mL$ & $2m^2$ \\\rule{0pt}{2.6ex}
$\delta$ bound & $1$ & $\frac{m}{\sqrt{2} L}$ & $\sqrt{\frac{m}{2L}}$ & $\sqrt{\frac{m}{L+m}}$ \\\hline
\end{tabular}
\caption{Upper bound on $\delta$ for the four different cases described in Table~\ref{tab:Mmat}.}
\label{tab:deltabound}
\end{table}

We can clearly see that for $\ell_2$-regularized logistic regression and support vector machines which admit the assumption in Case \rom{1}, biased SGD is robust to the relative noise. Given the condition $\delta<1$, the iterates of biased SGD will stay in some ball, although the size of the ball could be large. 
Comparing the  bound for Cases \rom{2}, \rom{3}, and \rom{4}, we can see the allowable relative noise level increases as the assumptions on $f_i$ become stronger.

As previously mentioned, the bound of Corollary~\ref{cor:linear2} requires a sufficiently small $\alpha$. 
Specifically, the stepsize $\alpha$ must satisfy $M_{21}\left(\alpha+\frac{\delta^2}{m}\right)\le 1$ and \eqref{rho_messy}, which can be solved to obtain explicit upper bounds on $\alpha$. Details are omitted.

\paragraph{Sensitivity analysis.}  Based on \eqref{eq:trho_approx},  the convergence rate $\bar{\rho}^2$ can be estimated as $1-\frac{m^2-\tilde{M}\delta^2}{m}\alpha$ for small $\alpha$, which is independent of $c$. Hence the misspecification in the value of $c$ does not impact the value of $\bar{\rho}^2$. The derivative of $1-\frac{m^2-\tilde{M}\delta^2}{m}\alpha$  with respect to $\delta$ is $2\tilde{M}\delta \alpha/m$. Hence, if we perturb the value of $\delta$ by $\epsilon$, the change in the value of $\bar{\rho}^2$ is roughly equal to $2\tilde{M}\delta \alpha\epsilon/m$. Similarly,  we can perform a sensitivity analysis for the final optimization error term $\bar{U}_\star$ by taking the derivative of the right side of \eqref{utilde} with respect to $\delta$ (or $c$).

\paragraph{Conservatism of Corollary \ref{cor:linear2}.}  Corollary \ref{cor:linear2} gives a reasonable non-recursive bound for biased SGD with small $\alpha$. However, if we consider Case \rom{2},  it can be much more conservative than Corollary \ref{cor:linear1} for relatively larger $\alpha$ . We use a numerical example to illustrate this. Consider $m=1$, $L=100$, $G=5$, and $c=1$. We set $\delta=0.0021<0.0071=m/\sqrt{\tilde{M}}$. Based on \eqref{alph2}, we know Corollary \ref{cor:linear1} works for $\alpha<7\times 10^{-5}$. Based on \eqref{rho_messy}, we can show Corollary \ref{cor:linear2} works for $\alpha<4.55\times 10^{-5}$. Obviously, Corollary \ref{cor:linear1} works for a larger range of $\alpha$. By numerical simulations, it is straightforward to verify that $\bar{U}_\star\rightarrow \infty$ and $\bar{\rho}\rightarrow 1$ if we apply Corollary \ref{cor:linear2} to the case where $\alpha= 4.56\times 10^{-5}$. In contrast, if we apply Corollary \ref{cor:linear1} to the case where $\alpha= 4.56\times 10^{-5}$,  we can obtain $\hat{U}_\star=4.8207$. The associated convergence rate is $1-1.74\times 10^{-5}$. Clearly, Corollary \ref{cor:linear1} gives a much more reasonable bound in this case. We have tried different problem parameters and  have observed similar trends. In general, for Case \rom{2},  Corollary  \ref{cor:linear2} is more conservative than Corollary \ref{cor:linear1} if relatively large $\alpha$ is considered. The advantage of Corollary \ref{cor:linear2} is that it is general enough to cover Cases \rom{1}, \rom{3}, and \rom{4}.

\section{Further discussion}
\label{sec:further}

\subsection{Connections to existing SGD results}

In this section, we relate the results of Theorem~\ref{thm:main} and its corollaries to existing results on standard SGD. We also discuss the effect of replacing our error model~\eqref{eq:def_ykhk0} with IID noise.

If there is no noise at all, $c=\delta=0$ and none of the approximations of Section~\ref{sec:fixed_alpha} are required to obtain an analytical bound on the iteration error. Returning to Theorem~\ref{thm:main} and Corollary~\ref{thm:cor1}, the objective to be minimized no longer depends on $\mu_k$. Examining~\eqref{eq:mainFor}, we conclude that optimality occurs as $\zeta\to 0$ ($\mu \to \infty$). This leads directly to the bound
\begin{equation}\label{eq:standardSGresult}
\ee \norm{x_{k+1}-x_\star}^2 \le (1-2m \alpha_k + \tilde M\alpha_k^2 )\ee\norm{x_k-x_\star}^2 + 2 G^2\alpha_k^2,
\end{equation}
where $\alpha_k$ is constrained such that $M_{21}\alpha_k \le 1$. The bound~\eqref{eq:standardSGresult} directly leads to existing convergence results for standard SGD. For example, we can apply the argument in Remark~\ref{rem:remark1} to obtain the following bound for standard SGD with a constant stepsize $\alpha_k=\alpha$
\begin{align}
\ee \|x_k-x_\star\|^2 \le \left(1-2m \alpha + \tilde M\alpha^2 \right)^k \ee\|x_0-x_\star\|^2 + \frac{2 G^2\alpha}{2m-\tilde M \alpha},
\end{align}
where $\alpha$ is further required to satisfy $1-2m\alpha+\tilde M \alpha^2\le 1$. For Cases \rom{1}, \rom{3}, and \rom{4}, the condition $M_{21}\alpha \le 1$ dominates, and
the valid values of $\alpha$ are documented in Corollary \ref{thm:cor1}. For Case \rom{2}, the condition $\alpha\le 2m/\tilde M$ dominates and the upper bound on $\alpha$ is $m/L^2$.

The bound recovers existing results that describe the design trade-off of standard (noiseless) SGD under a variety of conditions \cite{Nedic2001, moulines2011, needell2014}.
Case \rom{1} is a slight variant of the well-known result~\cite[Prop.~3.4]{Nedic2001}. The extra factor of $2$ in the rate and errors terms are due to the fact that \cite[Prop.~3.4]{Nedic2001} poses slightly different conditions on $g$ and $f_i$. Cases \rom{2} and \rom{3} are also well-known \cite{moulines2011, needell2014, feyzmahdavian2014}.

\begin{remark}
If the error term $e_k$ is IID noise with zero mean and bounded variance, then a slight modification to our analysis yields the bound
\begin{equation}\label{eq:standardSGIID}
\ee \norm{x_{k+1}-x_\star}^2 \le (1-2m \alpha_k + \tilde M\alpha_k^2 )\ee\norm{x_k-x_\star}^2 + (2 G^2+\sigma^2)\alpha_k^2,
\end{equation}
where $\sigma^2\ge \mathbb{E} \norm{e_k}^2$.
The detailed proof is omitted.
\end{remark}

\subsection{Adaptive stepsize via sequential minimization}

In Section~\ref{sec:fixed_alpha}, we fixed $\alpha_k=\alpha$ and derived bounds on the worst-case performance of biased SGD. In this section, we discuss the potential impacts of adopting time-varying stepsizes.
First,
we refine the bounds by optimizing over $\alpha_k$ as well. What makes this approach tractable is that in Theorem~\ref{thm:main}, the LMI~\eqref{eq:LMI_big} is also linear in $\alpha_k$. Therefore, we can easily include $\alpha_k$ as one of our optimization variables.

In fact, the development of Section~\ref{sec:sequential} carries through if we augment the set $\mathcal{T}_k$ to be the set of tuples $(\rho_k,\lambda_k,\mu_k,\nu_k,\alpha_k)$ that makes the LMI~\eqref{eq:LMI_big} feasible. We then obtain a Bellman-like equation analogous to~\eqref{yuck} that holds when we also optimize over $\alpha$ at every step. The net result is an optimization problem similar to~\eqref{eq:opt_simpler2} but that now includes $\alpha$ as a variable: 
\begin{equation}\label{eq:opt_simpler3}
\begin{aligned}
V_{k+1} \,\,=\,\, \minimize_{\alpha>0,\,\zeta > 0} \qquad &
a_k(1+\zeta^{-1}) + b_k(1+\zeta)\\
\text{subject to}\qquad
&a_k=\alpha^2 \left( c^2+2\delta^2 G^2+\tilde M\delta^2 V_k\right)\\
&b_k=\left(1-2m\alpha+\tilde M\alpha^2\right)V_k+2\alpha^2 G^2\\
&\alpha M_{21} (1+\delta^2\zeta^{-1})\le 1
\end{aligned}
\end{equation}
As we did in Section~\ref{sec:sequential}, we can show that $\ee\norm{x_k-x_\star}^2\le V_k$ for any iterates of biased SGD.
We would like to learn two things from~\eqref{eq:opt_simpler3}: how the optimal $\alpha$ changes as a function of $k$ in order to produce the fastest possible convergence rate, and whether this optimized rate is different from the rate we obtained when assuming $\alpha$ was constant in Section~\ref{sec:fixed_alpha}.

To simplify the analysis, we will restrict our attention to Case \rom{2}, where $M_{21}=0$ and $\tilde M = 2L^2$. In this case, the inequality constraint in~\eqref{eq:opt_simpler3} is satisfied for any $\alpha>0$ and $\zeta>0$, so it may be removed. Observe that the objective in~\eqref{eq:opt_simpler3} is a quadratic function of $\alpha$. 
\begin{multline}\label{eq:quad_alpha}
a_k(1+\zeta^{-1}) + b_k(1+\zeta)
= (1+\zeta) V_k
-2m (1+\zeta) V_k \alpha
+(1+\zeta^{-1})(c^2+2G^2\delta^2 \\+ \tilde M V_k \delta^2 + 2G^2\zeta + \tilde M V_k \zeta) \alpha^2
\end{multline}
This quadratic is always positive definite, and the optimal $\alpha$ is given by:
\begin{align}\label{alphaopt}
\alpha^\textup{opt}_k = \frac{m V_k \zeta }{(c^2+2 \delta ^2 G^2+\delta ^2 \tilde M V_k)+(2G^2 +\tilde M V_k)\zeta}
\end{align}
Substituting~\eqref{alphaopt} into~\eqref{eq:opt_simpler3} to eliminate $\alpha$, we obtain the optimization problem:
\begin{equation}\label{eq:opt_simpler4}
\begin{aligned}
V_{k+1} \,\,=\,\, \minimize_{\zeta > 0} \quad &
\frac{(\zeta +1) V_k \bigl(c^2+(2 G^2+\tilde M V_k) (\delta ^2+\zeta )-m^2V_k\zeta\bigr)}{c^2+ \bigl(2 G^2+\tilde M V_k\bigr)\left(\delta ^2+\zeta \right)}
\end{aligned}
\end{equation}
By taking the second derivative with respect to $\zeta$ of the objective function in~\eqref{eq:opt_simpler4}, one can check that we will have convexity as long as $(2G^2+\tilde M V_k)(1-\delta^2) \ge c^2$. In other words, we have convexity as long as the noise parameters $c$ and $\delta$ are not too large. If this bound holds for $V_k=0$, then it will hold for any $V_k>0$. So it suffices to ensure that $2G^2(1-\delta^2)\ge c^2$.

Upon careful analysis of the objective function, we note that when $\zeta=0$, we obtain $V_{k+1} = V_k$. In order to obtain a decrease for some $\zeta>0$, we require a negative derivative at $\zeta=0$. This amounts to the condition: $c^2 + (2G^2+\tilde M V_k)\delta^2 < m^2 V_k$. As $V_k$ gets smaller, this condition will eventually be violated. Specifically, the condition holds whenever     
$m^2-\tilde M \delta^2>0$ and
\[
V_k > \frac{c^2+2\delta^2 G^2}{m^2-\tilde M \delta^2}
\]
Note that this is the same limit as was observed in the constant-$\alpha$ limits $\hat U_\star$ and $\tilde U_\star$ when $\alpha \to 0$ in~\eqref{e:ustar} and \eqref{utilde}, respectively.  This is to be expected;  the biased gradient information introduces an uncontrollable bias (which is quantified as $\frac{c^2+2\delta^2 G^2}{m^2-\tilde{M}\delta^2}$)  into the final optimization result, and this 
can not be overcome by any stepsize rules.
Notice that we have not ruled out the possibility that $V_k$ suddenly jumps below $\frac{c^2+2\delta^2 G^2}{m^2-\tilde M \delta^2}$ at some $k$ and then stays unchanged after that. We will make a formal argument to rule out this possibility in the next lemma. 
Moreover, the question remains as to whether this minimal error can be achieved \textit{faster} by varying $\alpha_k$ in an optimal manner. We describe the final nonlinear recursion in the next lemma.

\begin{lemma}\label{lem:complicated_nonlinear}
Consider biased SGD~\eqref{eq:SGerror} with $g\in\svx(m,\infty)$ for some $m>0$, and let $x_\star$ be the unique global minimizer of $g$. Suppose Case \rom{2} holds and $(M, G)$ are the associated values from Table \ref{tab:Mmat}. Further assume $2G^2(1-\delta^2) \ge c^2$ and $V_0> \frac{c^2+2\delta^2 G^2}{m^2-\tilde M \delta^2}=V_\star$.

\begin{enumerate}
\item The sequential optimization problem \eqref{eq:opt_simpler4} can be solved using the following nonlinear recursion
\begin{multline}\label{ugly_nonlinear}
  V_{k+1} =\frac{V_k}{(2 G^2+\tilde M V_k)^2}\\
  \times \left(
\sqrt{(2 G^2+(\tilde M-m^2) V_k) \bigl((2 G^2+\tilde MV_k )(1-\delta ^2)-c^2\bigr)}\right. \\
+\left.\sqrt{m^2 V_k (c^2+\delta ^2 (2 G^2+\tilde M V_k))}
\right)^2
\end{multline}
and $V_k$ satisfies $V_k> V_\star$ for all $k$.

\item Suppose $ \hat U_0 = V_0 \ge \ee \norm{x_0-x_\star}^2$ (all recurrences are initialized the same way), then $\{V_k\}_{k\ge 0}$ provides an upper bound to the iterate error satisfying $\ee \norm{x_k-x_\star}^2 \le V_k \le \hat U_k$. 

\item The sequence $\{V_k\}_{k\ge 0}$ converges to $V_\star$:
\[
\lim_{k\to \infty} V_k =V_\star= \frac{c^2+2\delta^2 G^2}{m^2-\tilde M \delta^2}
\]
\end{enumerate}

\end{lemma}
\begin{proof}
See Appendix \ref{sec:proofLemma4}.
\end{proof}

To learn more about the rate of convergence, we can once again use a Taylor series approximation. Specializing to Case \rom{2} (where $\tilde M >0$), we can consider two cases. When $V_k$ is large, perform a Taylor expansion of~\eqref{ugly_nonlinear} about $V_k=\infty$ and obtain:
\[
V_{k+1} \approx \left(\tfrac{m\delta + \sqrt{(\tilde M-m^2)(1-\delta^2)}}{\tilde M}\right)V_k + O(1)
\]

In other words, we obtain linear convergence. When $V_k$ is close to $V_\star$, the behavior changes. To see this, perform a Taylor expansion of~\eqref{ugly_nonlinear} about $V_k=V_\star$ and obtain:
\begin{align}\label{rec_v}
V_{k+1} \approx V_k - \frac{(m^2-\tilde M \delta^2)^3}{4m^2\bigl( c^2(\tilde M-m^2) + 2G^2m^2(1-\delta^2)\bigr)} \left( V_k - V_\star \right)^2 + O((V_k-V_\star)^3)
\end{align}

We will ignore the higher-order terms, and apply the next lemma to show that the 
above recursion roughly converges at a $O(1/k)$ rate.

\begin{lemma}\label{lem:quadbound}
Consider the recurrence relation
\begin{equation}\label{q:1}
v_{k+1} = v_k - \eta v_k^2\qquad\text{for }k=0,1,\dots
\end{equation}
where $v_0>0$ and $0 < \eta < v_0^{-1}$. Then the iterates satisfy the following bound for all $k\ge 0$.
\begin{equation}\label{q:2}
v_k \le \frac{1}{\eta k + v_0^{-1}}
\end{equation}
\end{lemma}
\begin{proof}
The recurrence~\eqref{q:1} is equivalent to
$
\eta v_{k+1} = \eta v_k -(\eta  v_k)^2
$
with $0 < \eta v_0 < 1$. Clearly, the sequence $\{\eta v_k\}_{k\ge 0}$ is monotonically decreasing to zero. To bound the iterates, invert the recurrence:
\[
\frac{1}{\eta v_{k+1}} = \frac{1}{\eta v_k-(\eta v_k)^2} = \frac{1}{\eta v_k} + \frac{1}{1-\eta v_k} \ge \frac{1}{\eta v_k}+1
\]
Recursing the above inequality, we obtain:
$
\frac{1}{\eta v_{k}} \ge \frac{1}{\eta v_0} + k
$.
Inverting this inequality yields~\eqref{q:2}, as required.
\end{proof}

Applying Lemma~\ref{lem:quadbound} to the sequence $v_k = V_k-V_\star$ defined in~\eqref{rec_v}, we deduce that when $V_k$ is close to its optimal value of $V_\star$, we have:
\begin{equation}\label{vklinear}
V_k \sim V_\star + \frac{1}{\eta k + (V_0-V_\star)^{-1}}
\quad\text{with: }
\eta = \frac{(m^2-\tilde M \delta^2)^3}{4m^2\bigl( c^2(\tilde M-m^2) + 2G^2m^2(1-\delta^2)\bigr)}
\end{equation}
We can also examine how $\alpha_k$ changes in this optimal recursive iteration by taking~\eqref{alphaopt} and substituting the optimal~$\zeta$ found in the optimization of Lemma~\ref{lem:complicated_nonlinear}. The result is messy, but a Taylor expansion about $V_k=V_\star$ reveals that
\[
\alpha_k^\textup{opt} \approx
\frac{(m^2-\tilde M \delta^2)^2}{2m\bigl( c^2(\tilde M-m^2) + 2G^2m^2(1-\delta^2)\bigr)}(V_k-V_\star) + O((V_k-V_\star)^2).
\]
So when $V_k$ is close to $V_\star$, we should be decreasing $\alpha_k$ to zero at a rate of $O(1/k)$ so that it mirrors the rate at which $V_k-V_\star$ goes to zero in~\eqref{vklinear}.

 Finally, we want to mention that calculating $\alpha_k^\textup{opt}$ requires one to know the problem parameters $(m, \tilde{M}, \delta, c, G, V_0)$ in advance. This restricts the applicability of such adaptive stepsize rules for practical problems. Nevertheless, our results in this section bring new theoretical insights for the potential impacts of time-varying stepsizes on the performance of biased SGD. In summary, adopting an optimized time-varying stepsize still roughly yields a rate of $O(1/k)$, which is consistent with the sublinear convergence rate of standard SGD with diminishing stepsize.  It is possible that the well-known lower complexity bounds for standard SGD in \cite{agarwal2012information} can be extended to the inexact case, although a formal treatment is beyond the scope of this paper.

\section*{Appendix}
\appendix
\section{Proof of Theorem \ref{thm:main}} \label{proof:lemma2}

First notice that since $i_k$ is uniformly distributed on $\{1,\dots,n\}$ and $x_k$ and $i_k$ are independent, we have: 
\[
\ee \bigl( u_k \,\big|\, x_k \bigr) = 
\ee \bigl( \nabla f_{i_k}(x_k) \,\big|\, x_k \bigr) = \frac{1}{n}\sum_{i=1}^n \nabla f_i(x_k) = \nabla g(x_k)
\]
Consequently, we have:
\begin{align}
\label{eq:SGR2}
\ee \left(\bmtx x_k-x_\star \\ u_k \emtx^\tp\!
\addtolength{\arraycolsep}{-1pt}
\bmtx -2mI_p &  I_p \\ I_p & 0_p\emtx
\bmtx x_k-x_\star \\ u_k \emtx\middle|\  x_k\right) 
= \bmtx x_k-x_\star \\ \nabla g(x_k) \emtx^\tp\!
\addtolength{\arraycolsep}{-1pt}
\bmtx  -2m I_p & I_p\\ I_p & 0_p\emtx
\bmtx x_k-x_\star \\ \nabla  g(x_k) \emtx \ge 0
\end{align}
where the inequality in~\eqref{eq:SGR2} follows from the definition of $g \in \svx(m,\infty)$. 

Next we prove \eqref{eq:SQC3}, let's start with Case~\rom{1}, the boundedness constraint $\|\nabla f_i(x_k)\|\le \beta$ implies that $\norm{u_k} \le \beta$ for all $k$. Rewrite as a quadratic form to obtain:
\begin{align}
\label{eq:C1pr}
\bmtx x_k-x_\star\\u_k \emtx^\tp
\bmtx 0_p & 0_p \\ 0_p & -I_p \emtx
\bmtx x_k-x_\star\\u_k \emtx \ge -\beta^2
\end{align}
 The boundedness constraint $\|\nabla f_i(x_k)-mx_k\|\le \beta$ implies that:
\begin{align*}
\norm{u_k-m(x_k-x_\star)}^2
&\le \norm{(u_k-mx_k)+mx_\star}^2 + \norm{(u_k-mx_k)-mx_\star}^2\\
&=2\norm{u_k-mx_k}^2 + 2m^2\norm{x_\star}^2 \\
&\le 2\beta^2 + 2m^2\norm{x_\star}^2
\end{align*}
As in the proof of \eqref{eq:C1pr}, rewrite the above inequality as a quadratic form and we obtain the second row of Table~\ref{tab:Mmat}.

To prove the three remaining cases, we begin by showing that an inequality of the following form holds for each $f_i$:
\begin{align}
\label{eq:fiCons}
\bmtx x_k-x_\star \\ \nabla f_i(x_k)-\nabla f_i(x_\star) \emtx^\tp
\bmtx M_{11}I_p &  M_{12}I_p \\ M_{21}I_p & -2I_p\emtx
\bmtx x_k-x_\star \\ \nabla f_i(x_k)-\nabla f_i(x_\star) \emtx \ge 0
\end{align}
The verification for \eqref{eq:fiCons} follows directly from the definitions of $L$-smoothness and convexity. In the smooth case (Definition~\ref{def:smooth}), for example, $\|\nabla f_i(x_k)-\nabla f_i(x_\star)\|\le L\|x_k-x_\star\|$. So~\eqref{eq:fiCons} holds with $M_{11}=2L^2$, $M_{12}=M_{21}=0$. The cases for $\cvx(0,L)$ and $\cvx(m,L)$ follow directly from Definition~\ref{def:cvx}.
In Table~\ref{tab:Mmat}, we always have $M_{22}=-1$. Therefore,
\begin{multline}
\label{eq:SGfiConkey0}
\ee \left(\bmtx x_k-x_\star \\ u_k \emtx^\tp
\bmtx M_{11}I_p &  M_{12}I_p \\ M_{21}I_p & M_{22}I_p\emtx
\bmtx x_k-x_\star \\ u_k \emtx\,\,\middle|\ \, x_k\right) \\
=\frac{1}{n} \sum_{i=1}^n \bmtx x_k-x_\star \\ \nabla  f_i(x_k)\emtx^\tp \bmtx M_{11}I_p &  M_{12}I_p \\ M_{21}I_p & 0_p\emtx
\bmtx x_k-x_\star \\ \nabla  f_i(x_k)\emtx- \frac{1}{n} \sum_{i=1}^n \|\nabla f_i(x_k)\|^2
\end{multline}
Since $\frac{1}{n}\sum_{i=1}^n \nabla f_i(x_\star) = \nabla g(x_\star) = 0$, the first term on the right side of \eqref{eq:SGfiConkey0} is equal to
\begin{align*}
\frac{1}{n} \sum_{i=1}^n \bmtx x_k-x_\star \\ \nabla  f_i(x_k)-\nabla f_i(x_\star)\emtx \bmtx M_{11}I_p &  M_{12}I_p \\ M_{21}I_p & 0_p\emtx
\bmtx x_k-x_\star \\ \nabla  f_i(x_k)-\nabla f_i(x_\star) \emtx
\end{align*}
Based on the constraint condition~\eqref{eq:fiCons}, we know that the above term is greater than or equal to $\frac{2}{n} \sum_{i=1}^n \|\nabla  f_i(x_k)-\nabla f_i(x_\star) \|^2$. Substituting this fact back into \eqref{eq:SGfiConkey0} leads to the inequality:
\begin{align}
\label{eq:SGfiConkey1} \notag
\hspace{1cm}&\hspace{-1cm} \ee \left(\bmtx x_k-x_\star \\ u_k \emtx^\tp
\bmtx M_{11}I_p &  M_{12}I_p \\ M_{21}I_p & M_{22}I_p\emtx
\bmtx x_k-x_\star \\ u_k \emtx\,\,\middle|\ \, x_k\right) \\ \notag
&\ge \frac{1}{n} \sum_{i=1}^n \left(2\|\nabla  f_i(x_k)-\nabla f_i(x_\star) \|^2- \|\nabla f_i(x_k)\|^2\right) \\ \notag
&= \frac{1}{n} \sum_{i=1}^n \left( \|\nabla  f_i(x_k)-2\nabla f_i(x_\star) \|^2- 2\|\nabla f_i(x_\star)\|^2\right)\\
&\ge -\frac{2}{n} \sum_{i=1}^n \|\nabla f_i(x_\star)\|^2
\end{align}
Taking the expectation of both sides, we arrive at \eqref{eq:SQC3}, as desired.
Now we are ready to prove our main theorem.
By Schur complement, \eqref{eq:LMI_big} is equivalent to \eqref{eq:threebythree}, which can be further rewritten as
\begin{align}
\label{eq:keystep}
\left(\bmtx 1-\rho_k^2 & -\alpha_k & -\alpha_k\\[0.5mm]-\alpha_k & \alpha_k^2 & \alpha_k^2 \\[0.5mm]-\alpha_k & \alpha_k^2 & \alpha_k^2 \emtx
+\nu_k\!\bmtx -2m & 1 & 0\\1 & 0 & 0 \\0 & 0 &0 \emtx
+\lambda_k\!\bmtx M_{11} & M_{12} & 0\\M_{21} & M_{22} & 0 \\0 & 0 & 0 \emtx
+\mu_k\!\bmtx 0 & 0 & 0\\0 & \delta^2 & 0 \\0 & 0 &-1 \emtx\right)\otimes I_p\preceq 0
\end{align}
Since $x_{k+1}-x_\star=x_k-x_\star-\alpha_k(u_k+e_k)$, we have
\begin{align}
\bmtx x_k-x_\star \\ u_k \\ e_k \emtx^\tp \left(\bmtx 1 & -\alpha_k & -\alpha_k\\-\alpha_k & \alpha_k^2 & \alpha_k^2 \\-\alpha_k & \alpha_k^2 &\alpha_k^2 \emtx\otimes I_p\right) \bmtx x_k-x_\star \\ u_k \\ e_k \emtx=\|x_{k+1}-x_\star\|^2
\end{align}

Now we can left and right multiply \eqref{eq:keystep}  by $[(x_k-x_\star)^\tp,  u_k^\tp, e_k^\tp]$ and $[(x_k-x_\star)^\tp,  u_k^\tp, e_k^\tp]^\tp$, and apply the inequalities \eqref{eq:def_ykhk0}, \eqref{eq:SGR2}, and \eqref{eq:SQC3} to get
the desired conclusion.
\qed

\section{Proof of Lemma \ref{lem:complicated_nonlinear}}
\label{sec:proofLemma4}
We use an induction argument to prove Item 1. For simplicity, we denote \eqref{eq:opt_simpler4} as $V_{k+1}=h(V_k)$. Suppose we have $V_k=h(V_{k-1})$ and $V_{k-1}>V_\star$. We are going to show $V_{k+1}=h(V_k)$ and $V_k>V_\star$.
We can rewrite \eqref{eq:opt_simpler4} as
\begin{equation}
\label{eq:Vformu1}
V_{k+1}=\,\, \minimize_{\zeta > 0} \quad A_k(1+Z_k^{-1}) + B_k(1+Z_k)
\end{equation}
where  $A_k$, $B_k$, and $Z_k$ are defined as
\begin{align*}
A_k&=\frac{m^2 V_k^2 \left(c^2+(2G^2+\tilde M V_k)\delta^2\right)}{(2G^2+\tilde M V_k)^2}\\
B_k&=\frac{(2G^2 V_k+(\tilde M-m^2) V_k^2)((2G^2+\tilde M V_k)(1-\delta^2)-c^2)}{(2G^2+\tilde M V_k)^2}\\
Z_k&=\frac{\bigl(2 G^2+\tilde M V_k\bigr)(\delta ^2+\zeta_k )+c^2}{(2G^2+\tilde M V_k)(1-\delta^2)-c^2}
\end{align*}
Note that $A_k \ge 0$ and $B_k \ge 0$ due to the condition $2G^2(1-\delta^2) \ge c^2$. The objective in~\eqref{eq:Vformu1} therefore has a form very similar to the objective in~\eqref{eq:opt_simpler2}. Applying Lemma~\ref{lem:recursiveBound}, we deduce that $V_{k+1} =(\sqrt{A_k}+\sqrt{B_k})^2$, which is the same as~\eqref{ugly_nonlinear}.
The associated $Z_k^\textup{opt}$ is $\sqrt{\tfrac{A_k}{B_k}}$. To ensure this is a feasible choice, it remains to check that the associated $\zeta_k^\textup{opt} > 0$ as well. Via algebraic manipulations, one can show that $\zeta_k > 0$ is equivalent to $V_k > V_\star$. 
We can also verify $A_k$ is a monotonically increasing function of $V_k$, and $B_k$ is a monotonically nondecreasing function of $V_k$. Hence $h$ is a monotonically increasing function. Also notice $V_\star$ is a fixed point of \eqref{ugly_nonlinear}. 
Therefore, if we assume $V_k=h(V_{k-1})$ and $V_{k-1}>V_\star$, we have $V_k=h(V_{k-1})>h(V_\star)=V_\star$. Hence we guarantee $\zeta_k>0$ and $V_{k+1}=h(V_k)$. By similar arguments, one can verify $V_1=h(V_0)$. And it is assumed that $V_0>V_\star$. This completes the induction argument.  

Item 2 follows from a similar argument to the one used in Section~\ref{sec:sequential}. Finally, Item 3 can be proven by choosing a sufficiently small constant stepsize $\alpha$ to make $\hat U_k$ arbitrarily close to $V_\star$. Since $V_\star \le V_k \le \hat U_k$, we conclude that $\lim_{k\to\infty} V_k = V_\star$, as required. \qed


\bibliographystyle{abbrv}
\bibliography{IQCandSOS}

\begin{thebibliography}{10}

\bibitem{agarwal2012information}
A.~Agarwal, P.~L. Bartlett, P.~Ravikumar, and M.~J. Wainwright.
\newblock Information-theoretic lower bounds on the oracle complexity of
  stochastic convex optimization.
\newblock {\em IEEE Transactions on Information Theory}, 58(5):3235--3249,
  2012.

\bibitem{arora2015}
S.~Arora, R.~Ge, T.~Ma, and A.~Moitra.
\newblock Simple, efficient, and neural algorithms for sparse coding.
\newblock In {\em Conference on Learning Theory}, pages 113--149, 2015.

\bibitem{berts2002}
D.~Bertsekas.
\newblock {\em Nonlinear Programming}.
\newblock Belmont: Athena scientific, 2nd edition, 2002.

\bibitem{bottou2010}
L.~Bottou.
\newblock Large-scale machine learning with stochastic gradient descent.
\newblock In {\em Proceedings of COMPSTAT'2010}, pages 177--186. 2010.

\bibitem{bottou2018optimization}
L.~Bottou, F.~E. Curtis, and J.~Nocedal.
\newblock Optimization methods for large-scale machine learning.
\newblock {\em SIAM Review}, 60(2):223--311, 2018.

\bibitem{Bottou2004}
L.~Bottou and Y.~LeCun.
\newblock Large scale online learning.
\newblock {\em Advances in Neural Information Processing Systems}, 16:217,
  2004.

\bibitem{bubeck2015}
S.~Bubeck.
\newblock Convex optimization: Algorithms and complexity.
\newblock {\em Foundations and Trends{\textregistered} in Machine Learning},
  8(3-4):231--357, 2015.

\bibitem{chen2015}
Y.~Chen and E.~Candes.
\newblock Solving random quadratic systems of equations is nearly as easy as
  solving linear systems.
\newblock In {\em Advances in Neural Information Processing Systems}, pages
  739--747, 2015.

\bibitem{d2008smooth}
A.~d'Aspremont.
\newblock Smooth optimization with approximate gradient.
\newblock {\em SIAM Journal on Optimization}, 19(3):1171--1183, 2008.

\bibitem{de2017worst}
E.~De~Klerk, F.~Glineur, and A.~Taylor.
\newblock On the worst-case complexity of the gradient method with exact line
  search for smooth strongly convex functions.
\newblock {\em Optimization Letters}, 11(7):1185--1199, 2017.

\bibitem{defazio2014}
A.~Defazio, F.~Bach, and S.~Lacoste-Julien.
\newblock Saga: A fast incremental gradient method with support for
  non-strongly convex composite objectives.
\newblock In {\em Advances in Neural Information Processing Systems}, 2014.

\bibitem{defazio2014finito}
A.~Defazio, J.~Domke, and T.~Caetano.
\newblock Finito: A faster, permutable incremental gradient method for big data
  problems.
\newblock In {\em Proceedings of the 31st International Conference on Machine
  Learning}, pages 1125--1133, 2014.

\bibitem{devolder2014first}
O.~Devolder, F.~Glineur, and Y.~Nesterov.
\newblock First-order methods of smooth convex optimization with inexact
  oracle.
\newblock {\em Mathematical Programming}, 146(1-2):37--75, 2014.

\bibitem{drori2014}
Y.~Drori and M.~Teboulle.
\newblock Performance of first-order methods for smooth convex minimization: a
  novel approach.
\newblock {\em Mathematical Programming}, 145(1-2):451--482, 2014.

\bibitem{feyzmahdavian2014}
H.~Feyzmahdavian, A.~Aytekin, and M.~Johansson.
\newblock A delayed proximal gradient method with linear convergence rate.
\newblock In {\em 2014 IEEE International Workshop on Machine Learning for
  Signal Processing}, pages 1--6, 2014.

\bibitem{cvx2}
M.~Grant and S.~Boyd.
\newblock Graph implementations for nonsmooth convex programs.
\newblock In V.~Blondel, S.~Boyd, and H.~Kimura, editors, {\em Recent Advances
  in Learning and Control}, Lecture Notes in Control and Information Sciences,
  pages 95--110. Springer-Verlag Limited, 2008.
\newblock \url{http://stanford.edu/~boyd/graph_dcp.html}.

\bibitem{cvx1}
M.~Grant and S.~Boyd.
\newblock {CVX}: Matlab software for disciplined convex programming, version
  2.1.
\newblock \url{http://cvxr.com/cvx}, Mar. 2014.

\bibitem{hu17b}
B.~Hu, P.~Seiler, and A.~Rantzer.
\newblock A unified analysis of stochastic optimization methods using jump
  system theory and quadratic constraints.
\newblock In {\em Proceedings of the 2017 Conference on Learning Theory},
  volume~65, pages 1157--1189, 2017.

\bibitem{johnson2013}
R.~Johnson and T.~Zhang.
\newblock Accelerating stochastic gradient descent using predictive variance
  reduction.
\newblock In {\em Advances in Neural Information Processing Systems}, pages
  315--323, 2013.

\bibitem{lee2016optimizing}
J.~C. Lee and P.~Valiant.
\newblock Optimizing star-convex functions.
\newblock In {\em 2016 IEEE 57th Annual Symposium on Foundations of Computer
  Science (FOCS)}, pages 603--614, 2016.

\bibitem{Lessard2014}
L.~Lessard, B.~Recht, and A.~Packard.
\newblock Analysis and design of optimization algorithms via integral quadratic
  constraints.
\newblock {\em SIAM Journal on Optimization}, 26(1):57--95, 2016.

\bibitem{moulines2011}
E.~Moulines and F.~Bach.
\newblock Non-asymptotic analysis of stochastic approximation algorithms for
  machine learning.
\newblock In {\em Advances in Neural Information Processing Systems}, pages
  451--459, 2011.

\bibitem{Nedic2001}
A.~Nedi{\'c} and D.~Bertsekas.
\newblock Convergence rate of incremental subgradient algorithms.
\newblock {\em Stochastic optimization: algorithms and applications}, pages
  223--264, 2001.

\bibitem{needell2014}
D.~Needell, R.~Ward, and N.~Srebro.
\newblock Stochastic gradient descent, weighted sampling, and the randomized
  kaczmarz algorithm.
\newblock In {\em Advances in Neural Information Processing Systems}, pages
  1017--1025, 2014.

\bibitem{nishihara2015}
R.~Nishihara, L.~Lessard, B.~Recht, A.~Packard, and M.~Jordan.
\newblock A general analysis of the convergence of {ADMM}.
\newblock In {\em Proceedings of the 32nd International Conference on Machine
  Learning}, pages 343--352, 2015.

\bibitem{robbins1951}
H.~Robbins and S.~Monro.
\newblock A stochastic approximation method.
\newblock {\em The Annals of Mathematical Statistics}, 22(3):400--407, 1951.

\bibitem{Roux2012}
N.~Roux, M.~Schmidt, and F.~Bach.
\newblock A stochastic gradient method with an exponential convergence rate for
  strongly-convex optimization with finite training sets.
\newblock In {\em Advances in Neural Information Processing Systems}, 2012.

\bibitem{Schmidt2013}
M.~Schmidt, N.~Le~Roux, and F.~Bach.
\newblock Minimizing finite sums with the stochastic average gradient.
\newblock {\em Mathematical Programming}, 162(1-2):83--112, 2017.

\bibitem{schmidt2011}
M.~Schmidt, N.~L. Roux, and F.~R. Bach.
\newblock Convergence rates of inexact proximal-gradient methods for convex
  optimization.
\newblock In {\em Advances in Neural Information Processing Systems}, pages
  1458--1466, 2011.

\bibitem{shalev2013}
S.~Shalev-Shwartz and T.~Zhang.
\newblock Stochastic dual coordinate ascent methods for regularized loss.
\newblock {\em The Journal of Machine Learning Research}, 14(1):567--599, 2013.

\bibitem{sun2016guaranteed}
R.~Sun and Z.-Q. Luo.
\newblock Guaranteed matrix completion via non-convex factorization.
\newblock {\em IEEE Transactions on Information Theory}, 62(11):6535--6579,
  2016.

\bibitem{taylor19a}
A.~Taylor and F.~Bach.
\newblock Stochastic first-order methods: non-asymptotic and computer-aided
  analyses via potential functions.
\newblock In {\em Proceedings of the 2019 Conference on Learning Theory}, pages
  2934--2992, 2019.

\bibitem{taylor2017}
A.~Taylor, J.~Hendrickx, and F.~Glineur.
\newblock Smooth strongly convex interpolation and exact worst-case performance
  of first-order methods.
\newblock {\em Mathematical Programming}, 161(1-2):307--345, 2017.

\bibitem{taylor2017exact}
A.~Taylor, J.~M. Hendrickx, and F.~Glineur.
\newblock Exact worst-case performance of first-order methods for composite
  convex optimization.
\newblock {\em SIAM Journal on Optimization}, 27(3):1283--1313, 2017.

\bibitem{taylor18a}
A.~Taylor, B.~Van~Scoy, and L.~Lessard.
\newblock {L}yapunov functions for first-order methods: Tight automated
  convergence guarantees.
\newblock In {\em Proceedings of the 35th International Conference on Machine
  Learning}, pages 4897--4906, 2018.

\bibitem{teo2007}
C.~Teo, A.~Smola, S.~Vishwanathan, and Q.~Le.
\newblock A scalable modular convex solver for regularized risk minimization.
\newblock In {\em Proceedings of the 13th ACM SIGKDD international conference
  on Knowledge discovery and data mining}, pages 727--736, 2007.

\end{thebibliography}

\end{document}